\documentclass[11pt,letterpaper]{amsart}

\usepackage{amsmath, amscd, amssymb}
\usepackage{graphpap, color}
\usepackage[mathscr]{eucal}
\usepackage{cancel}
\usepackage{verbatim}
\usepackage{adjustbox}
\usepackage{tikz}
\usepackage{float}
\usepackage{booktabs}
\usepackage{multirow}
\usepackage{tabularx}
\usepackage{bbm}
\usepackage{stmaryrd}
\usepackage{mathtools}
\usepackage{subcaption}

\usepackage{listings}

\numberwithin{equation}{section}

\newtheorem{theorem}{Theorem}[section]
\newtheorem{corollary}[theorem]{Corollary}

\newtheorem{conjecture}[theorem]{Conjecture}

\theoremstyle{definition}
\newtheorem{definition}[theorem]{Definition}
\newtheorem{question}[theorem]{Question}
\newtheorem{example}[theorem]{Example}
\newtheorem{remark}[theorem]{Remark}

\newcommand{\C}{\mathbb{C}}

\newcommand{\PP}{\mathbb{P}}

\def\CC{\mathbb{C}}

\def\RR{\mathbb{R}}

\def\sO{{\mathscr O}}

\newcommand{\cal}{\mathcal}

\def\cM{{\cal M}}
\def\cN{{\cal N}}
\def\sO{{\cal O}}
\def\cP{{\cal P}}
\def\cQ{{\cal Q}}

\def\fp{\mathfrak{p}}

\def\hbar{\overline{h}}

\def\Mbar{\overline{\cM}}
\def\mzn{\overline{\cM}_{0,n} }
\def\mzno{\overline{\cM}_{0,n+1} }

\def\SL{\mathrm{SL} }

\def\dim{\mathrm{dim} }

\def\ch{\mathrm{ch} }

\def\multiset#1#2{\ensuremath{\left(\kern-.3em\left(\genfrac{}{}{0pt}{}{#1}{#2}\right)\kern-.3em\right)}}
\def\multisetBig#1#2{\ensuremath{\Big(\kern-.3em\Big(\genfrac{}{}{0pt}{}{#1}{#2}\Big)\kern-.3em\Big)}}
\def\multisetbody#1#2{\ensuremath{\big(\kern-.3em\big(\genfrac{}{}{0pt}{}{#1}{#2}\big)\kern-.3em\big)}}

\def\and{\quad{\rm and}\quad}

\def\beq{\begin{equation}}
\def\eeq{\end{equation}}
\def\ben{\begin{enumerate}}
\def\een{\end{enumerate}}

\def\and{\quad\text{and}\quad}

\def\sfs{\mathsf{s}}

\def\stan{\mathscr{S}}

\def\charP{\mathscr{P}}
\def\charQ{\mathscr{Q}}
\def\symS{\mathbb{S}}

\def\fmn{\overline{\mathcal{M}}_{0,n}(\PP^1,1)}

\title{Asymptotic distribution of the Betti numbers of $\overline{\mathcal{M}}_{0,n}$}

\author{Jinwon Choi}
\address{Department of Mathematics and Research Institute of Natural Sciences, Sookmyung Women's University, Seoul 04310, Korea}
\email{jwchoi@sookmyung.ac.kr}

\author{Young-Hoon Kiem}
\address{School of Mathematics, Korea Institute for Advanced Study, 85 Hoegiro, Dongdaemun-gu, Seoul 02455, Korea}
\email{kiem@kias.re.kr}

\date{\today}

\begin{document}

\begin{abstract}
Asymptotic normality is frequently observed in large combinatorial structures, rigorously established for many quantities such as cycles or inversions in random permutations, the number of prime factors of random integers, and various parameters of random graphs. In this paper, we investigate whether this normal limit behavior extends to the topological invariants of geometric spaces. We show that the Betti numbers of the moduli space of rational curves with $n$ marked points $\mzn$ and the Fulton-MacPherson configuration space $\PP^1[n]$ are asymptotically normally distributed. Based on numerical evidence and established log-concavity, we conjecture that the Betti numbers of the quotients of these spaces by the symmetric group $\symS_n$ are also asymptotically normally distributed. In contrast, we provide examples of geometric spaces that do not follow this Gaussian law.
\end{abstract}

\maketitle

\section{Introduction}

The moduli space $\mzn$ of stable rational curves with $n$ marked points (cf. \cite{DM, Knu}) plays a central role in algebraic geometry and mathematical physics, serving as a fundamental object in subjects ranging from Gromov–Witten theory and quantum cohomology to operads and conformal field theory (cf. \cite{Manin}). Despite the significance of a complete understanding of its topological invariants, the overall topology of $\mzn$ is still far from fully understood. While Keel \cite{Keel} successfully described the cohomology ring and found recursive formulas for the Betti numbers, the combinatorial complexity of these invariants grows exponentially as $n$ gets large. Consequently, obtaining explicit descriptions for the Betti numbers for arbitrary $n$ is a highly nontrivial task.

To tackle this complexity, we turn our attention to the \emph{asymptotic behavior} of these invariants.  We investigate the limiting shape of the Betti numbers as the number of marked points tends to infinity. This approach is motivated by the phenomenon of \emph{asymptotic normality} in analytic combinatorics, where many combinatorial quantities tend to follow a normal distribution in the limit.

Using methods from analytic combinatorics, we analyze the generating functions of the Poincar\'e polynomials of $\mzn$. Our main result establishes that the sequence of Betti numbers of $\mzn$ is asymptotically normally distributed, meaning informally that, after appropriate normalization, its graph exhibits a shape close to that of a normal density function. This implies that the topology of $\mzn$ has a statistically predictable nature in the limit. Despite the exponential growth of the total cohomology, the distribution of its cohomology cycles mimics the behavior of sums of independent random variables. Using this result, we also establish the asymptotic log-concavity of the Betti sequence, strengthening prior work.

We further investigate whether this phenomenon of asymptotic normality extends to other fundamental moduli spaces. In particular, we demonstrate that the Fulton-MacPherson configuration space $\PP^1[n]$ exhibits asymptotic normality. We conjecture that the quotients $\mzn/\symS_n$, $\mzno/\symS_n$ and $\PP^1[n]/\symS_n$ also exhibit asymptotic normality. In contrast, the Hilbert scheme of points on a smooth projective surface and the GIT quotient $(\PP^1)^n/\!\!/\SL_2$ do not.
Consequently, the asymptotic normality is not a universal feature of all moduli spaces, but rather depends on the analytic nature of their generating functions. This distinction offers a novel approach to the classification and understanding of the topology of moduli spaces.

\medskip

\noindent \textbf{Acknowledgment.}
We thank Donggun Lee for his collaboration \cite{CKL1,CKL2,CKL3}. YHK thanks Professor Yujiro Kawamata for encouragement and Gergely B\'erczi for useful discussion.

\section{Preliminaries}
Throughout this paper, we focus on nonempty topological spaces $X_n$ whose total cohomology is finite-dimensional and odd-degree Betti numbers vanish. Let $b_{k,n} =\dim H^{2k}(X_n)$ be the even degree Betti number and let
\[ P_{X_n}(u) = \sum_k b_{k,n} u^k \]
be the corresponding Poincar\'{e} polynomial. The Euler characteristic is given by $\chi(X_n)=P_{X_n}(1)$. Since $\chi(X_n)>0$, the normalized polynomial 
\[f_n(u):=\frac{P_{X_n}(u)}{\chi(X_n)}=\sum_{k} \PP(\xi_n=k)u^k\]
can be viewed as a \emph{probability generating function} (PGF) of an associated integer-valued random variable $\xi_n$. We consider the asymptotic behavior of $\xi_n$ as $n\to \infty$.

For a probability generating function $f$ of an integer-valued random variable, it is straightforward that the mean and the variance are given by
\[\mathfrak{m}(f)=f'(1), \ \ \ \mathfrak{v}(f)= f''(1)+f'(1)-f'(1)^2 .\]
For a random variable $\xi$, its \emph{distribution function} $F_{\xi}(x)$ is defined by
\[F_{\xi}(x)=\PP(\xi \le x). \]
\begin{definition}
   Let $\xi$ be a continuous random variable with distribution function $F_{\xi}(x)$. A sequence of random variables $\xi_n$ with distribution function $F_{\xi_n}(x)$ is said to converge in distribution to $\xi$ if for each $x$,
    \[\lim_{n\to \infty} F_{\xi_n}(x) =F_{\xi}(x).\]
   In this case, we say the speed of convergence is $\epsilon_n$ if
   \[ \sup_{x\in \RR} |F_{\xi_n}(x) - F_{\xi}(x)|\le \epsilon_n. \]
\end{definition}

We say that the Betti numbers of $X_n$ are \emph{asymptotically normally distributed} if the corresponding standardized random variable $(\xi_n-\mathfrak{m}(f_n))/\sqrt{\mathfrak{v}(f_n)}$ converges in distribution to the standard normal distribution $\cN(0,1)$.

\begin{example}
Let $X$ be a topological space with finite-dimensional total cohomology and $H^{\text{odd}}(X)=0$. Let $f(u)=\frac{P_X(u)}{\chi(X)}$ be its normalized Poincar\'e polynomial. By the K\"unneth theorem, we have
\[ P_{X^n }(u) = P_X(u)^n, \  \chi(X^n)= \chi(X)^n. \]
So, the normalized Poincar\'e polynomial is $f_n(u)=f(u)^n$. Let $\xi$ and $\xi_n$ be the random variables associated to the PGF $f(u)$ and $f_n(u)$ respectively. Then $\xi_n$ is clearly the sum of $n$ independent copies of $\xi$. Hence by the Central Limit Theorem, the Betti numbers of $X^n$ are asymptotically normally distributed.
\end{example}

The key idea of the above example is generalized in the Quasi-Powers Theorem, which asserts that if the probability generating function essentially behaves like a power of a fixed function, then the corresponding coefficients are asymptotically normally distributed. This condition indicates that the underlying combinatorial invariant behaves like a sum of independent (or weakly dependent) random variables. In this sense, the Quasi-Powers Theorem can be regarded as a generalization of the classical Central Limit Theorem. In analytic combinatorics, asymptotic normality of the distribution encoded by $f_n(u)$ is often derived from the singularity analysis of the bivariate generating function \[\varphi(z, u)= \sum_{n\ge 0} f_n(u)z^n .\]
For a bivariate generating function $\varphi(z, u)$, we denote by $[z^n]\varphi(z,u)$ the coefficient of $z^n$ in $\varphi(z,u)$.

We state the Quasi-Powers Theorem. A comprehensive treatment can be found in \cite{FS}.

\begin{theorem}[{\cite[Theorem IX.8]{FS}}, Quasi-Powers Theorem] \label{thm:QPT}
    Let $\xi_n$ be non-negative integer-valued random variables with probability generating function $p_n(u)$. Suppose that there exist analytic functions $A(u)$, $B(u)$, independent of $n$ with $A(1)=B(1)=1$ and sequences $\beta_n, \kappa_n\to \infty$ such that
     \[p_n(u) =A(u)B(u)^{\beta_n} \left(1+O\left(\frac{1}{\kappa_n}\right)\right),\]
   uniformly in a fixed neighborhood $\Omega$ of $u=1$. Assume furthermore the variability condition
   \[ \mathfrak{v}(B(u))=B''(1) +B'(1)-B'(1)^2 \ne 0\]
   holds. Then, $\xi_n$ is asymptotically normally distributed. More precisely, with
   \beq\label{eq:mv}
   \begin{split}
    m_n &= \beta_n \mathfrak{m}(B(u))+\mathfrak{m}(A(u))+O(\kappa_n^{-1}) \\
    \sigma^2_n &= \beta_n \mathfrak{v}(B(u))+\mathfrak{v}(A(u))+O(\kappa_n^{-1}),
    \end{split}
    \eeq
    the random variable $\frac{\xi_n-m_n}{\sigma_n}$ converges in distribution to the standard normal distribution $\cN(0,1)$. The speed of convergence is $O(\kappa_n^{-1} + \beta_n^{-1/2})$.

\end{theorem}

Whereas the Quasi-Powers Theorem provides distributional asymptotic normality, obtaining estimates for individual coefficients requires a local limit theorem.

\begin{theorem}[{\cite[Theorem IX.14]{FS}}, Local limit theorem] \label{thm:LLT}
    Assume the hypothesis of the Quasi-Powers Theorem for $p_n(u)=\sum_k p_{k,n}u^k$. Assume in addition that
    \[ |p_n(u)|\le K^{-\beta_n} \]
    for some $K>1$ and all $u$ in $\CC-\Omega$ with $|u|=1$.  Then,
    \[ \lim_{n\to \infty} \sup_{x\in \RR} \left| \sigma_n p_{\lfloor m_n +x\sigma_n \rfloor,n} - \frac{1}{\sqrt{2\pi}}e^{-x^2/2} \right| =0.\]
\end{theorem}

\begin{example}
   Another motivating example is the complete flag variety $\mathrm{Fl}_n$. It is well known that the Poincaré polynomial of $\mathrm{Fl}_n$ coincides with the generating function for the number of inversions in permutations of the symmetric group $\symS_n$ (cf. \cite{Fulton}). As established in \cite{Hwang, Mar} via the Quasi-Powers Theorem, the distribution of inversions converges to a normal distribution as $n$ tends to infinity. Consequently, the Betti numbers of the complete flag varieties are asymptotically normally distributed as $n$ grows.
\end{example}

\section{Moduli space of rational curves with $n$ marked points}\label{sec:mzn}

Let $\mzn$ be the moduli space of rational curves with $n$ marked points. We define $b_{k,n} =\dim H^{2k}(\mzn)$ and
\beq\label{0}\varphi(z,u)=z+\sum_{n\ge 2}\frac{z^n}{n!}\sum_{k=0}^{n-2}b_{k, n+1} u^k. \eeq

It is well known that $\varphi$ satisfies the functional equation (cf. \cite[p.194 Theorem 4.2.1]{Manin})
\beq \label{eq:varphi} (1+\varphi)^u = u^2 \varphi -u(u-1)z +1.
\eeq
Hence, for a fixed value of $u$ (assuming $u \ne 0, 1$), $\varphi$ can be regarded as the inverse of
\beq\label{1} z=\frac{u^2\varphi-(1+\varphi)^u+1}{u(u-1)}.\eeq

By the Inverse Function Theorem, singularities of $\varphi$ can only occur when the derivative $ \frac{\partial z}{\partial \varphi}$ vanishes:
\[\frac{\partial z}{\partial \varphi}= \frac{u^2-u(1+\varphi)^{u-1}}{u(u-1)}=0.  \]
Thus, we find that the location of singularity is
\beq\label{eq:sing} z= \rho(u) := u^{\frac{1}{u-1}}- \frac{u+1}{u}, \hspace{1em} \varphi = \lambda(u):= u^\frac{1}{u-1}-1. \eeq
The second derivative at the singularity \eqref{eq:sing} is
\beq\label{eq:phi''} \frac{\partial^2 z}{\partial \varphi^2}|_{\varphi=\lambda(u)}= - u^{\frac{u-2}{u-1}}.  \eeq
Since this is nonzero, the singularity is of the square-root type (cf. \cite[\S IX.7.3.]{FS}).
By the Taylor expansion, we have
\[
\begin{split}
   z-\rho(u) & =  -\frac12 u^{\frac{u-2}{u-1}} (\varphi - \lambda(u))^2+O\left(|\varphi-\lambda(u)|^3\right) \\
    & = -\frac12 u^{\frac{u-2}{u-1}} (\varphi - \lambda(u))^2\left(1+O\left(|\varphi-\lambda(u)|\right)\right).
\end{split}
 \]
Inverting this local relation gives
\beq \label{eq:phiexp}
\begin{split}
\varphi - \lambda(u) & =  - \sqrt{2} u^{-\frac{u-2}{2(u-1)}} \rho(u)^\frac12 \left(1 - \frac{z}{\rho(u)}\right)^\frac12 \\ &
+ C(u)\left(1 - \frac{z}{\rho(u)}\right)+O\left(\left|1 - \frac{z}{\rho(u)}\right|^{\frac{3}{2}}\right)
\end{split}
\eeq
for some $C(u)$. By \cite[Theorem VI.4]{FS}, the approximation of $\varphi$ near the singularity transfers to an asymptotic estimate of the coefficients:
\begin{align}
  [z^n]\varphi(z, u)  & = -\sqrt{2} u^{-\frac{u-2}{2(u-1)}} \rho(u)^{\frac12 -n} (-1)^n \binom{\frac{1}{2}}{n} + O\left(\rho(u)^{-n}n^{-\frac{5}{2}}\right)  \notag \\
   & = \frac{1}{\sqrt{2\pi}}u^{-\frac{u-2}{2(u-1)}}\rho(u)^{\frac12 -n} n^{-\frac{3}{2}} + O\left(\rho(u)^{-n}n^{-\frac{5}{2}}\right). \label{eq:varphiasymp}
\end{align}
Note that the linear error term $(1-z/\rho(u))$ in \eqref{eq:phiexp} does not contribute to the asymptotics of coefficients. The factor $n^{-\frac{5}{2}}$ is from the next contributing error term $(1-z/\rho(u))^{\frac{3}{2}}$. Thus, we obtain
\beq\label{eq:asym}
\frac{[z^n]\varphi(z, u)}{[z^n]\varphi(z, 1)}   = \frac{u^{-\frac{u-2}{2(u-1)}}}{\sqrt{e}}\left(\frac{\rho(u)}{e-2}\right)^{\frac12-n}\left( 1 + O\left(n^{-1}\right)\right).
\eeq

Applying Theorem \ref{thm:QPT} to \eqref{eq:asym} with $$A(u)=\frac{u^{-\frac{u-2}{2(u-1)}}}{\sqrt{e}}\left(\frac{\rho(u)}{e-2}\right)^{\frac12}, \ \  B(u)=\frac{e-2}{\rho(u)},$$ we arrive at the following theorem.

\begin{theorem}\label{thm:mzn}
  The Betti numbers of $\mzn$ are asymptotically normally distributed with the mean $m_n=\frac{n-3}{2}$ and the variance   \beq\label{eq:mznvar} \sigma_n^2= \frac{3-e}{6(e-2)}n+ \frac{11-4e}{24-12e} + O(n^{-1}).\eeq
  The speed of convergence is $O(n^{-\frac12}).$
\end{theorem}

\begin{example}

\begin{figure}
    \centering
    \includegraphics[width=0.7\columnwidth]{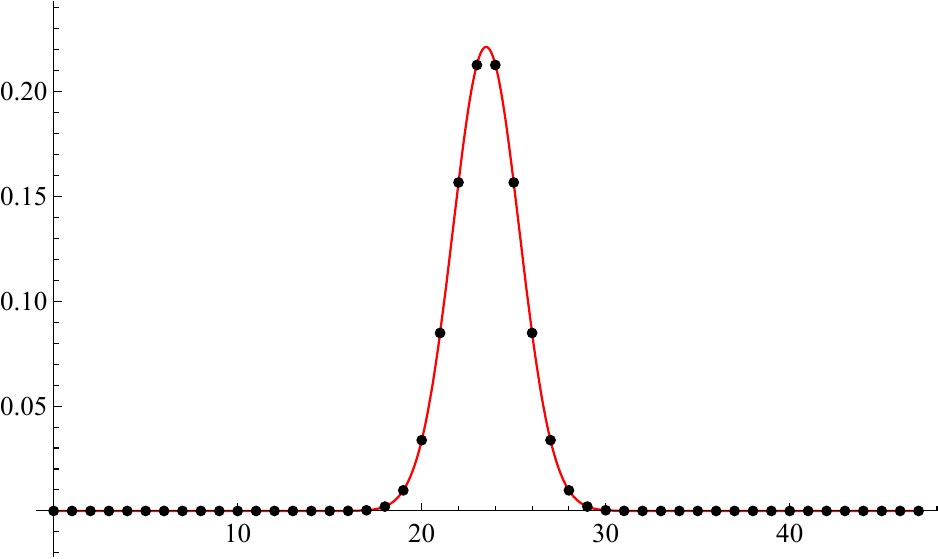}
    \caption{Betti distribution for $\overline{\cM}_{0,50}$}
    \label{fig:mzn}
\end{figure}

Figure \ref{fig:mzn} shows the normalized Betti numbers (black dots) of $\overline{\cM}_{0,50}$ together with the probability density function of the corresponding normal distribution (red curve). The close agreement indicates that the normal density function provides an effective approximation of the Betti numbers, particularly for those in the middle degrees. For instance, the relative estimation error for the middle Betti number $b_{23,50}$ is approximately $0.13\%$.
\end{example}

An application of Theorem \ref{thm:mzn} is that the Betti sequence of $\mzn$ is asymptotically log-concave, which extends the results in \cite{ACM,CKL2}.

\begin{definition}
\begin{enumerate}
  \item A sequence $\{a_k\}_{k=0}^n$ is called \emph{log-concave} if \[a_k^2 \ge a_{k-1}a_{k+1}\] holds for all $1 \le k \le n-1$.
  \item A sequence $\{a_k\}_{k=0}^n$ is called \emph{$r$-ultra-log-concave} at $k$ if \[\left(\frac{a_k}{\binom{n}{k}^r} \right)^2\ge \left(\frac{a_{k-1}}{\binom{n}{k-1}^r } \right)\left(\frac{a_{k+1}}{\binom{n}{k+1}^r} \right)\] holds.
\end{enumerate}
\end{definition}

The variance in \eqref{eq:mznvar} of the distribution of the Betti numbers of $\mzn$ grows asymptotically at the linear rate $v= \frac{3-e}{6(e-2)} \approx 0.06537$. In comparison, the variance of the binomial coefficients $\binom{n}{k}$, interpreted as a probability distribution, grows at the rate $\frac{1}{4}$.
We derive the following result from the inequality $v < \frac{1}{4r}$ for $r=3$. We restrict our attention to the central region, since the Gaussian approximation used in Theorem \ref{thm:mzn} is valid in this region by the local limit theorem (Theorem \ref{thm:LLT}).

\begin{corollary}
  The sequence of Betti numbers of $\mzn$ is asymptotically $3$-ultra-log-concave at $k$ for $|k-\frac{n-3}{2}|=O(\sqrt{n})$.
\end{corollary}
\begin{proof}
  We first check that the local limit theorem holds for Betti numbers of $\mzn$. Let \[\varphi_n(u)=[z^{n}]\varphi(z, u)=\sum_{k=0}^{n-2}b_{k, n+1} u^k.\] By previous discussion, we find that the radius of convergence of $\varphi(u)$ is $|\rho(u)|$ and it is elementary to check that $|\rho(u)|$ attains its minimum uniquely at $u=1$ for $|u|=1$. 
  
  Let $\Omega$ be a small neighborhood of $u=1$ and $\Gamma=\{u\in \CC ~|~ |u|=1\}\setminus\Omega$. Let $M=\min_{u\in \Gamma}|\rho(u)|$. By the Cauchy–Hadamard theorem, we have \[ \lim_{n\to \infty} \left|\frac{\varphi_n(u)}{\varphi_n(1)}\right|^{-\frac{1}{n}} \ge \frac{M}{e-2}, \]
  uniformly for $u \in \Gamma$. Hence $\frac{\varphi_n(u)}{\varphi_n(1)}$ satisfies the uniform bound condition of Theorem \ref{thm:LLT} with $K=\frac{M}{e-2}$.

  Hence, for $k$ with $|k-\frac{n-3}{2}|=O(\sqrt{n})$, the normalized Betti numbers satisfy
  \beq\label{eq:bettiasymp} \frac{b_{k,n}}{\chi(\mzn)} \sim \frac{1}{\sqrt{2\pi}\sigma_n} \exp\left({-\frac{(k-m_n)^2}{2\sigma_n^2}}\right). \eeq
  Similarly, by the de Moivre-Laplace theorem, the approximation of binomial coefficients
  \[ \binom{n-3}{k} 2^{-(n-3)} \sim \frac{1}{\sqrt{2\pi s_n^2}} \exp\left({-\frac{(k-m_n)^2}{2 s_n^2}}\right), \]
  where $s_n^2 = \frac{n-3}{4}$, is valid for $|k-m_n|=o(n^{\frac{2}{3}})$. In particular, this includes the central region $|k-m_n|=O(\sqrt{n})$.

  Thus, asymptotically
  \[  \frac{b_{k,n}}{\binom{n-3}{k}^r} \sim C_n \exp\left( -\frac{(k-m_n)^2}{2} \left( \frac{1}{\sigma_n^2}-\frac{r}{s_n^2} \right) \right), \]
 where $C_n$ is a positive constant depending only on $n$. This sequence (in $k$) is log-concave if and only if the coefficient of the quadratic term $(k-m_n)^2$ is negative, equivalently, the resulting distribution is a normal distribution with positive variance. This requires
  \[ \frac{1}{\sigma_n^2} - \frac{r}{s_n^2} > 0 \;\Leftrightarrow \; \frac{s_n^2}{\sigma_n^2} > r. \]
 Since $\sigma_n^2 \sim vn$ and $s_n^2 \sim \frac{n}{4}$, the condition is $$r < \frac{1}{4v}.$$ Using the value $v \approx 0.06537$, we have $\frac{1}{4v} \approx 3.82$. Thus, the condition holds for $r=3$.
\end{proof}

\begin{remark}
We note that earlier proofs of asymptotic log-concavity \cite{ACM, CKL2} are most effective in the tail regions of the Betti sequence, since they fix $k$ and let $n\to \infty$. Our approach provides a new verification that extends this property to the central region, where the index $k$ grows proportionally with $n$.
\end{remark}

A byproduct of the preceding proof is the asymptotic formula \eqref{eq:bettiasymp} for Betti numbers $b_{k,n}$ when $k$ varies with $n$. This complements the results of \cite{ACM}, where an asymptotic formula is obtained for fixed $k$. For instance, when $n$ is even, \eqref{eq:bettiasymp} reads
\[ \frac{b_{\frac{n-2}{2}, n+1}}{\chi(\mzno)} \sim  \frac{1}{\sqrt{2\pi\sigma_{n+1}^2 }}\sim  \frac{1}{\sqrt{\frac{3-e}{3(e-2)}\pi n}}.\]
By combining with asymptotic formula for $\chi(\mzno)$ (cf. \cite[IV. Theorem.4.2.1]{Manin}), we obtain
\[ b_{\frac{n-2}{2}, n+1}\sim \frac{1}{\sqrt{\frac{3-e}{3(e-2)}\pi}} \frac{1}{n}\left(\frac{n}{e^2-2e} \right)^{n-\frac{1}{2}}. \]

\section{Fulton-MacPherson space of $n$ points in $\PP^1$}\label{sec:FM}

Let $\PP^1[n]=\fmn$ denote the Fulton-MacPherson compactification of the configuration space of $n$ distinct ordered points in $\PP^1$. Let
\beq\label{2} \psi(z,u)=1+\sum_{n\ge 1}\frac{z^n}{n!} \sum_k u^k\dim H^{2k}(\PP^1[n]).\eeq

An explicit relation between $\psi$ and $\varphi$ is given in \cite[IV. (4.24)]{Manin}\footnote{In fact, a stronger version is proved in \cite{Manin}. For $m=1$, the solution $y^0$ of \cite[IV. (4.22)]{Manin} is $\varphi$ by \cite[IV. (4.7)]{Manin}. 
}:
\beq\label{4} \psi=(1+\varphi)^{u+1}.\eeq
Alternatively, \eqref{4} can also be derived from \cite[Proposition 6.8]{CKL1}\footnote{There is a typo in \cite[Proposition 6.8]{CKL1}. The $m-2$ on the right-hand side should be $m-1$. Note also that when passing from Frobenius characteristics to dimensions of representations, a combinatorial factor $\binom{n}{a}$ arises in the summation.}. Indeed, from that proposition, it is easy to see that
\beq\label{3}
\psi= (1+\varphi)(u^2 \varphi -u(u-1)z +1).
\eeq
Combining this with \eqref{eq:varphi} yields \eqref{4}.

By differentiating with respect to $z$, we find that
\beq\label{5}
\frac{\partial \psi}{\partial z}=(u+1)(1+\varphi)^u\frac{\partial \varphi}{\partial z}.\eeq
As the factor $(u+1)(1+\varphi)^u$ is nonzero near $u=1$, we find that 
\beq\label{6}\frac{\partial z}{\partial \psi}=0 \;\Leftrightarrow \;
\frac{\partial z}{\partial \varphi}=0 \;\Leftrightarrow \;
z=\rho(u),\; \varphi=\lambda(u), \; \psi=\gamma(u)\eeq
in a neighborhood of $u=1$, where
\beq\label{7}
\rho(u)=u^{\frac{1}{u-1}}-\frac{u+1}{u},\; \lambda(u)=u^{\frac{1}{u-1}}-1,\; \gamma(u)=u^{\frac{u+1}{u-1}}.\eeq
Moreover, we have
\beq\label{8}
\begin{split}
\frac{\partial^2z}{\partial \psi^2}|_{\psi=\gamma(u)}&=
\frac1{(u+1)(1+\lambda(u))^u}\frac{\partial^2 z}{\partial \varphi^2}\frac{\partial \varphi}{\partial \psi}|_{\psi=\gamma(u)}\\
&=-\frac{1}{(u+1)^2}u^{-\frac{u+2}{u-1}}
\end{split}
\eeq
where the last equality follows from \eqref{eq:phi''} and \eqref{4}.
By the Taylor expansion at \eqref{7}, we have
\beq\label{9}
z-\rho(u) =-\frac{u^{-\frac{u+2}{u-1}}}{2(u+1)^2}
(\psi-\gamma(u))^2+O\left(|\psi-\gamma(u)|^3\right).\eeq
Then we may apply \cite[Theorem VI.4]{FS} and proceed as in the case of $\varphi$ to find that
\beq \label{eq:FMasymp} \frac{[z^n]\psi(z,u)}{[z^n]\psi(z,1)}=
\frac{(u+1) u^\frac{u+2}{2(u-1)}}{2e\sqrt{e}} \left(\frac{\rho(u)}{e-2}\right)^{\frac12 -n}  \left(1+O(n^{-1})\right).\eeq
In the Appendix, we provide an alternative derivation of \eqref{eq:FMasymp} by establishing a simple relation between the Poincar\'e polynomials of $\mzn$ and $\PP^1[n]$. Now, by Theorem \ref{thm:QPT}, we have the following.
\begin{theorem}
 The Betti numbers of $\PP^1[n]$ are asymptotically normally distributed with the mean $\frac{n}{2}$ and the variance $$\frac{3-e}{6(e-2)}n+ \frac{15-7e}{24-12e} + O(n^{-1}).$$ The speed of convergence is $O(n^{-\frac12}).$
\end{theorem}

%
Since the variance grows linearly at the same rate $v$ as in the case of $\mzn$, the same argument as before yields the following.
\begin{corollary}
  The sequence of Betti numbers of $\PP^1[n]$ is asymptotically $3$-ultra-log-concave at $k$ for $|k-\frac{n}{2}|=O(\sqrt{n})$.
\end{corollary}

\begin{figure}
    \centering
    \includegraphics[width=0.7\columnwidth]{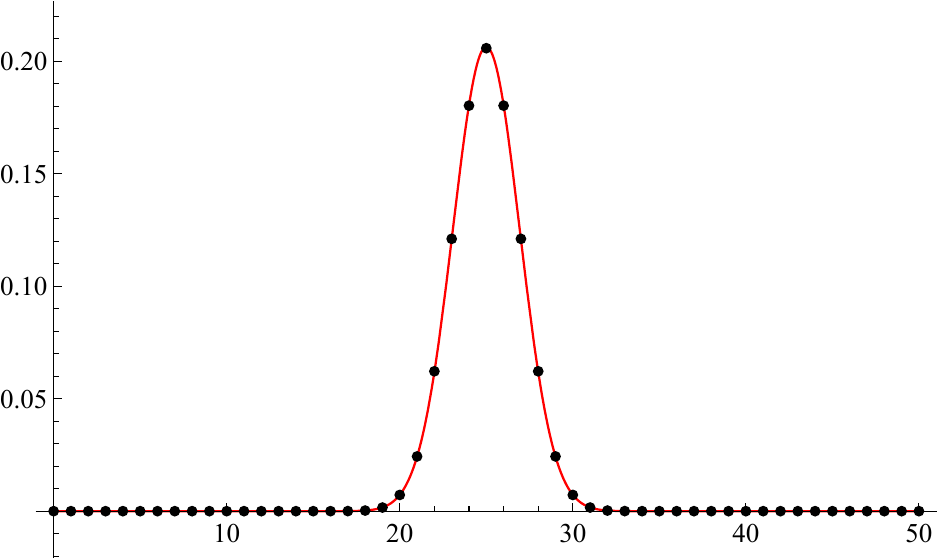}
    \caption{Betti distribution for $\PP^1[50]$}
    \label{fig:FM}
\end{figure}

Figure \ref{fig:FM} displays the Betti number distribution of $\PP^1[50]$ alongside the corresponding normal density, showing close agreement.

\section{More examples}
In this section, we present examples from geometry, some of which exhibit asymptotic normality and some of which do not. We aim to address the following question in future work.

\begin{question}
 Are there geometric conditions that ensure asymptotic normality?
\end{question}

As the asymptotic normality is often observed in combinatorial setting, we expect that it arises when the Poincar\'e polynomials of the geometric objects are governed by underlying combinatorial data. For example, Poincar\'e polynomials of $\mzno$ and $\mzno/\symS_n$ are encoded by the combinatorics of the weighted rooted trees (cf. \cite{CKL1, CKL2}).

\subsection{Spaces of rational curves with unordered marked points}

The symmetric group $\symS_n$ acts naturally on the moduli spaces $\mzn$, $\mzno$ and $\PP^1[n]$ by permuting the marked points (fixing the last marked point in the case of $\mzno$). The Betti numbers of the quotient spaces $\mzn/\symS_n$, $\mzno/\symS_n$ and $\PP^1[n]/\symS_n$ are determined by the multiplicities of the trivial representation in the $\symS_n$-representation on the corresponding cohomology (cf. \cite{CKL1}). Alternatively, without computing the full representations, they can also be obtained by evaluating the characteristic polynomial at $m=1$. We refer the reader to \cite{CKL3} for the definitions and algorithms used to compute these characteristic polynomials, and a brief overview is provided in the Appendix. We have carried out the computations for $n$ up to $72$.

\begin{figure}
    \centering
    \includegraphics[width=0.7\columnwidth]{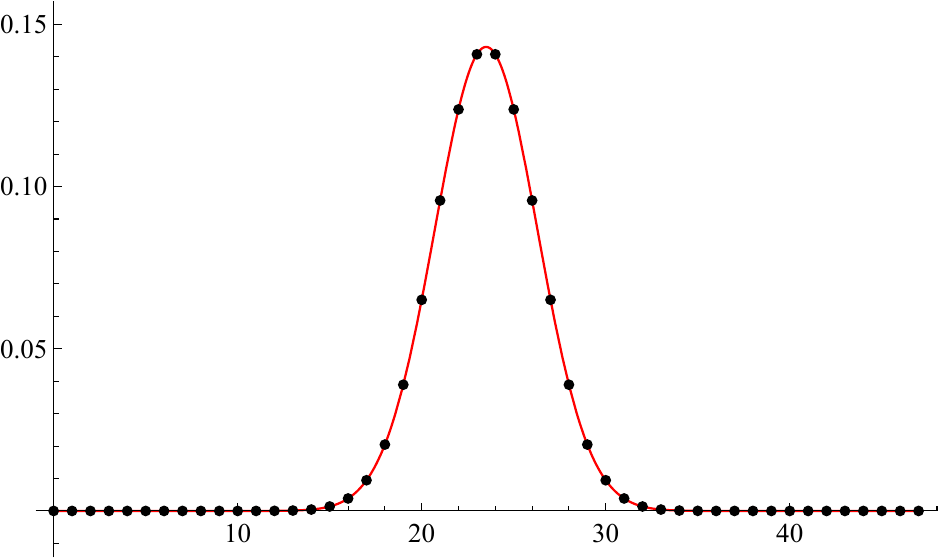}
    \caption{Betti distribution for $\overline{\cM}_{0,50}/\symS_{50}$}
    \label{fig:inv1}
\end{figure}
\begin{figure}
    \centering
    \includegraphics[width=0.7\columnwidth]{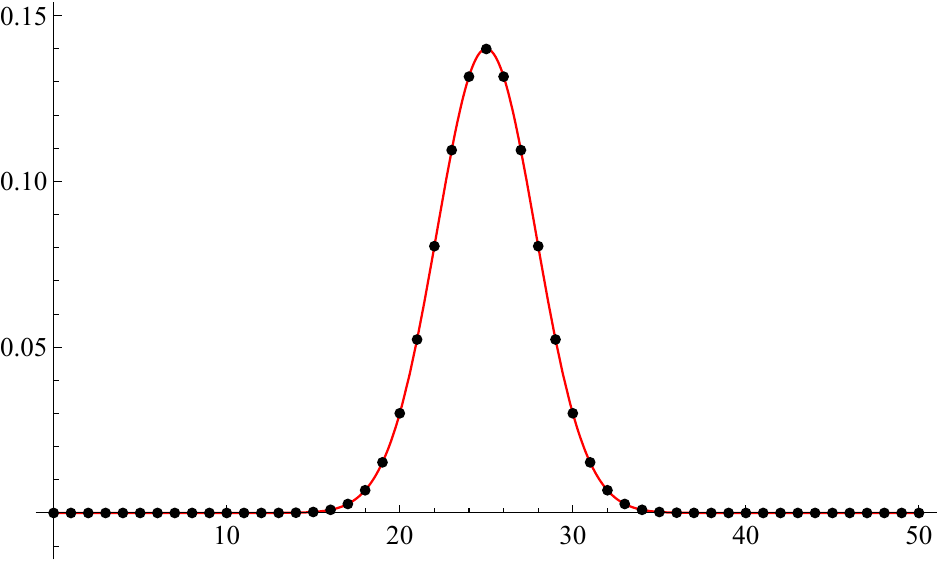}
    \caption{Betti distribution for $\PP^1[50]/\symS_{50}$}  
    \label{fig:inv2}
\end{figure}

As shown in \cite{CKL2}, the Betti numbers of $\mzn/\symS_n$ are asymptotically log-concave. This property was later extended to the characteristic polynomial in \cite{CKL3}, and it provides strong evidence for the asymptotic normality of the Betti numbers. Figures \ref{fig:inv1} and \ref{fig:inv2} show the normalized Betti numbers alongside the corresponding normal density function for $n=50$.

Furthermore, the variance of the Betti numbers exhibits linear growth with respect to $n$, which is similar to the ordered case we have seen in the previous sections. We present the values of the variance divided by $n$ in Table \ref{tab:variance}.

\begin{table}[htbp]
 \centering
  \begin{tabular}{cccc}

    $n$ & $\sigma^2(\mzn/\symS_{n})/n $ & $\sigma^2(\mzno/\symS_{n})/n $ & $\sigma^2(\PP^1[n]/\symS_{n})/n $    \\
    \hline
    10 & 0.1828947368  &0.1639097744 &  0.2066585956\\
    30 & 0.1593209245 &0.1554856327&  0.1700357511 \\
    50 &  0.1555469935 &0.1536360445& 0.1623732407 \\
    70 &  0.1541260042 &0.1528624719 & 0.1591027379 \\
    \hline
  \end{tabular}
  \caption{The normalized variance for the Betti numbers}
  \label{tab:variance}
\end{table}

Based on the numerical data, we propose the following conjecture.
\begin{conjecture}
    The Betti numbers of $\mzn/\symS_n$, $\mzno/\symS_n$ and $\PP^1[n]/\symS_{n}$ are asymptotically normally distributed and the variances associated with these three spaces grow linearly at the same rate.
\end{conjecture}
We expect that the arguments similar to those in Section \ref{sec:mzn} will apply. Assuming this, a proof demonstrating that the variances grow linearly at the same rate is provided in the Appendix. Determining the exact limiting value of the linear growth rate would be an interesting problem.


\subsection{Hilbert scheme of $n$ points on a surface}

Another important space in the moduli theory is the Hilbert scheme of $n$ points on a smooth projective surface $S$, denoted by $\mathrm{Hilb}^n(S)$. The Hilbert scheme is a smooth resolution of the symmetric product $\mathrm{Sym}^n(S)$. The Betti numbers of $\mathrm{Hilb}^n(S)$ are determined by G\"{o}ttsche's celebrated generating function formula \cite{Gottsche}. In contrast to $\mzn$, $\PP^1[n]$ and their quotients discussed above, we do not expect the Betti numbers of the Hilbert scheme to exhibit asymptotic normality. This is due to the analytic structure of its generating function, which prevents the application of the Quasi-Powers Theorem.

Let $S$ be either $\PP^2$ or $\PP^1\times \PP^1$ 
and let $b_i=\dim H^{2i}(S)$. Let $P_{\mathrm{Hilb}^n(S)}(u)=\sum_k u^k\dim H^{2k}(\mathrm{Hilb}^n(S))$ be the Poincar\'e polynomial. Recall that if the odd degree Betti numbers of $S$ vanish, the generating function $ \sum_{n=0}^\infty P_{\mathrm{Hilb}^n(S)}(u) z^n  $ is given by an infinite product:
\beq \label{eq:gottsche}
\prod_{m=1}^\infty \frac{1}{(1-u^{m-1}z^m)^{b_0}(1-u^{m}z^m)^{b_1}(1-u^{m+1}z^m)^{b_2}}. \eeq

To apply the Quasi-Powers Theorem via singularity analysis, the generating function is typically required to have an isolated singularity for $u$ near 1. (cf. \cite[Theorem VI.4]{FS}) However, due to the infinite product structure of \eqref{eq:gottsche}, the singularities accumulate, preventing the quasi-powers framework from being applicable.
Figures \ref{fig:hilb1} and \ref{fig:hilb2} plot the Betti distribution for $S=\PP^2$ and $S=\PP^1\times \PP^1$ against the normal density functions with the same mean and variance, which shows clear non-Gaussian behavior.
\begin{remark}
It is observed in \cite{HR} that the distribution of Betti numbers of $\mathrm{Hilb}^n(\CC^2)$ asymptotically converges to the \emph{Gumbel distribution}. In contrast to the normal distribution which models averages, the Gumbel distribution models extreme values (maximums) and consequently  it is asymmetric. This aligns with the fact that $\mathrm{Hilb}^n(\CC^2)$ does not satisfy Poincar\'{e} duality. It would be of interest to identify the precise limiting distribution for $S=\PP^2$ or $S=\PP^1\times \PP^1$.
\end{remark}

\begin{figure}
    \centering
    \includegraphics[width=0.7\columnwidth]{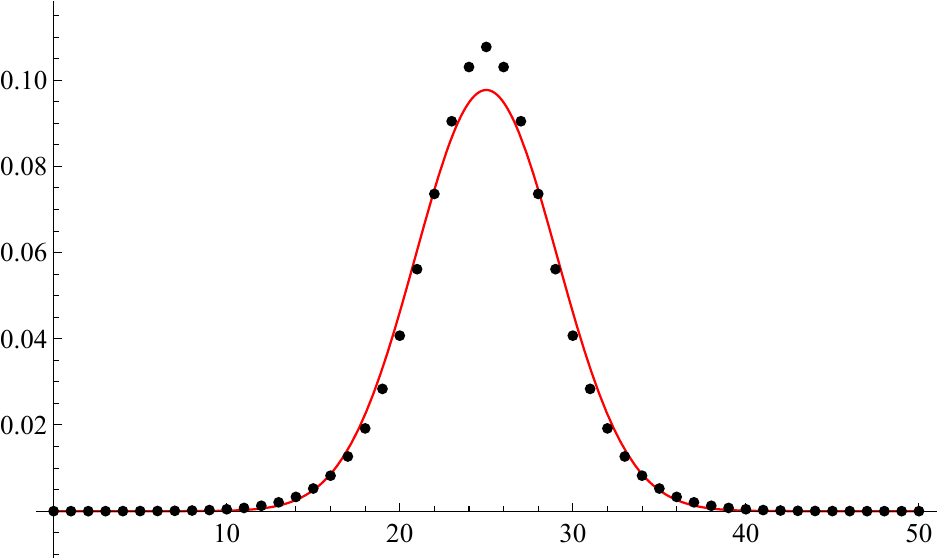}
    \caption{Betti distribution for $\mathrm{Hilb}^{25}(\PP^2)$}
    \label{fig:hilb1}
\end{figure}
\begin{figure}
    \centering
    \includegraphics[width=0.7\columnwidth]{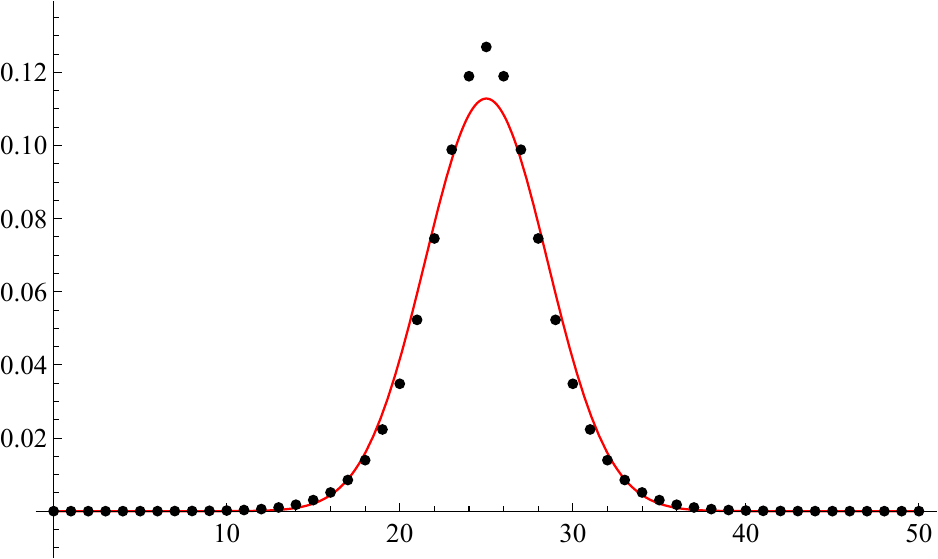}
    \caption{Betti distribution for $\mathrm{Hilb}^{25}(\PP^1\times \PP^1)$}
    \label{fig:hilb2}
\end{figure}

\subsection{GIT quotient of $(\PP^1)^n$}

Let $$Y_n=(\PP^1)^{n}/\!\!/\SL_2(\C)$$ be the GIT quotient (cf.~\cite{MFK}) of $(\PP^1)^n$ by the diagonal $\SL_2(\C)$-action with respect to the linearization $\sO_{(\PP^1)^n}(1,\cdots, 1)$. In this section, we restrict our attention to odd $n$ so that all semistable points are stable and that $Y_n$ is a smooth projective variety. As demonstrated in \cite{KM}, $\mzn$ is related to $Y_n$ by a sequence of blowups
\beq\label{eq:blowup} \mzn \to \cdots \to Y_n.  \eeq
Motivated by this connection, we explore whether the asymptotic normality observed in $\mzn$ persists for the Betti numbers of $Y_n$. The Betti numbers of $Y_n$ for odd $n$ are computed in \cite[Example~5.18]{Kir}:
\[
\ \sum_k u^k\dim H^{2k}(Y_n)= \sum_{k=0}^{n-3} t^k \left( \sum_{j=0}^{\min(k, n-3-k)}\binom{n-1}{j} \right).
\]

We observe that this distribution is not asymptotically normal. The Betti numbers are given by a partial sum of binomial coefficients. Since the sequence of binomial coefficients approaches a normal density function, their partial sums approximate the cumulative distribution function, while the symmetry is imposed by the condition on the summation. Figure \ref{fig:GIT} plots the Betti distribution for $Y_{51}$ against the normal density function with the same mean and variance, illustrating the clear deviation from Gaussian behavior.

\begin{figure}
    \centering
    \includegraphics[width=0.7\columnwidth]{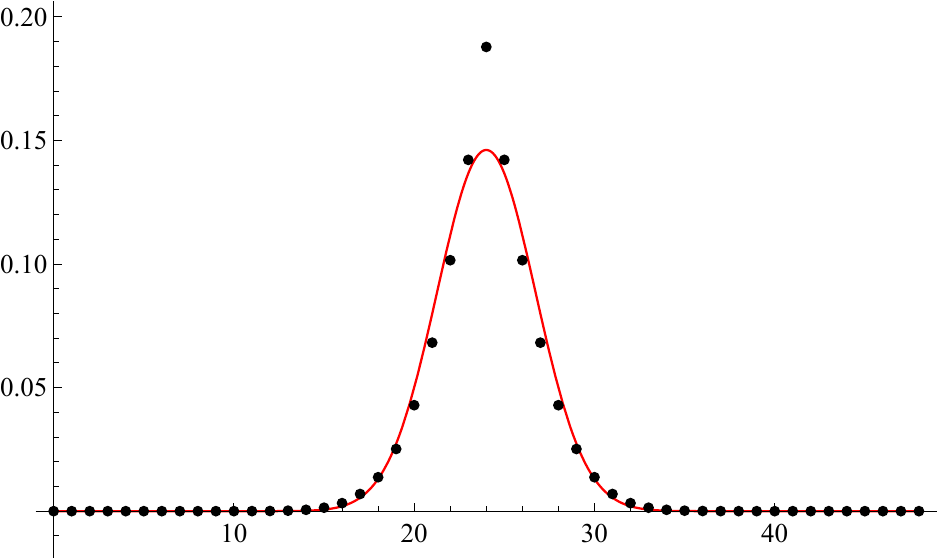}
    \caption{Betti distribution for $(\PP^1)^{51}/\!\!/\SL_2(\C)$}
    \label{fig:GIT}
\end{figure}

Since $\mzn$ is obtained from $Y_n$ via the sequence of blowups \eqref{eq:blowup}, it would be interesting to find out at which stage the asymptotic normality emerges and to identify the geometric reason for this transition. We leave this question for future investigation.

\section*{Appendix}
In the Appendix, we present a proof of the identities relating the Poincar\'e polynomials of $\mzn$, $\mzno$, $\PP^1[n]$, and their quotient spaces.
For a space $X_n$ equipped with an action of the symmetric group $\symS_n$, we define
\[ P_{X_n} =\sum_{k=0}^{\dim X_n} \dim H^{2k}(X_n)t^k, \ \cP_{X_n} =\sum_{k=0}^{\dim X_n} \ch_{\symS_n} H^{2k}(X_n)t^k
\]
to be the Poincar\'{e} polynomial of $X_n$ and the graded Frobenius characteristic of the $\symS_n$-representations on the cohomology of $X_n$, respectively.

Let \[\fp_{X_n}=\sum_{k=0}^{\dim X_n} \dim H^{2k}(X_n/\symS_n)t^k \] denote the Poincar\'{e} polynomial of the quotient $X_n/\symS_n$. This polynomial can be obtained by extracting the multiplicities of the trivial representation from $\cP_{X_n}$. In \cite{CKL3}, we define and study the characteristic polynomial $\stan_V(m)$ of a $\symS_n$-representation $V$, which is obtained by applying the \emph{Stanley map}
\beq \label{eq:stan} x_1=x_2=\cdots=x_m=1 \text{ and } x_{m+1}=x_{m+2}=\cdots = 0\eeq
to the Frobenius characteristic $\ch_{\symS_n} V$, where the variables $x_i$ correspond to the power-sum variables. Then $\stan_V(m)$ is a polynomial in $m$ of degree $n$ and $\stan_V(1)$ is the dimension of the invariant subspace $V^{\symS_n}$ (\cite[Lemma 3.10]{CKL3}).

Define the characteristic polynomial $\charP_{X_n}$ of $X_n$ by applying the Stanley map \eqref{eq:stan} to $\cP_{X_n}$. In particular, $\fp_{X_n}$ is obtained by evaluating $\charP_{X_n}$ at $m=1$.

In this note, we let $X_n$ be one of $\mzn$, $\mzno$ or $\PP^1[n]$, where the action of $\symS_n$ is given by permuting $n$ points (fixing the last marked point in the case of $\mzno$). We denote
\[P_n=P_{\mzn},  \ \ Q_n=P_{\mzno},\ \  P^{FM}_n=P_{\PP^1[n]}\]
and use the analogous notations for $\cP$, $\charP$ and $\fp$. We adopt the convention that these polynomials are zero whenever the index is non-integer.

In \cite{CKL2}, we study the relationship among the cohomologies of these spaces via quasi-map wall crossing. By \cite[Theorem 4.8]{CKL2}, we have
\beq\label{eq:QP} \cQ_n = (1+t)\cP_n + t \sfs_{(1,1)}\circ \cQ_{\frac{n}{2}} + t \sum_{h=2}^\ell \cQ_h \cQ_{n-h},
\eeq
where $\ell = \lfloor \frac{n-1}{2}\rfloor$, $\sfs_\lambda$ is the Schur function associated to the partition $\lambda$ and $\circ$ denotes the plethysm product.

There is an analogous formula for $\PP^1[n]$ (cf. Propositions 4.1 and 4.4 for $m=2$ of \cite{CKL2}):
\beq\label{eq:FM}
\begin{split}
\cP_{\PP^1[n]} &= (1+t+t^2+t^3)\cP_n + t \sfs_{(1,1)}\circ \left( (1+t) \cQ_{\frac{n}{2}}\right)\\
& + (t+t^2) \sum_{h=2}^\ell (1+t)\cQ_h \cQ_{n-h} + t (1+t)\sfs_1 \cQ_{n-1}.
\end{split}
\eeq
Since
\[
\sfs_{(1,1)}\circ \left( (1+t) \cQ_{\frac{n}{2}}\right)= \sfs_{(1,1)}\circ\cQ_{\frac{n}{2}} + t^2 \sfs_{(1,1)}\circ\cQ_{\frac{n}{2}}+t \cQ_{\frac{n}{2}}^2,
\]
by combining \eqref{eq:QP} and \eqref{eq:FM}, we find
\beq\label{eq:FMrep}
\begin{split}
 \cP_{\PP^1[n]} & =   (1+t)^2\cQ_n -2t(1+t)\cP_n \\ & \hspace{2ex}-2t^2 \sfs_{(1,1)}\circ  \cQ_{\frac{n}{2}} +
t^2 \cQ_{\frac{n}{2}}^2 + t(1+t) \sfs_1 \cQ_{n-1}.
\end{split}
\eeq

Consequently, we obtain a formula for the characteristic polynomials. Note that the characteristic polynomial of $\sfs_{(1,1)}\circ  \cQ_{\frac{n}{2}}$ is $\frac{1}{2} (\charQ_{\frac{n}{2}}^2 -\charQ_{\frac{n}{2}}^{[2]} )$, where the superscript $[2]$ denotes the substitution $t\mapsto t^2$. Thus,
\begin{align}
\charP_{\PP^1[n]} &=  (1+t)^2\charQ_n -2t(1+t)\charP_n + t(1+t) m \charQ_{n-1}+ t^2
\charQ_{\frac{n}{2}}^{[2]} \label{eq:FMchar} \\
&= (1+t^2)\charQ_n + t(1+t)m \charQ_{n-1}+ t^2\sum_{h=2}^{n-2}
\charQ_{h}\charQ_{n-h}  \label{eq:FMchar2}
\end{align}
The second equality follows from \cite[Corollary 5.12]{CKL1}:
\[(1+t)\charP_n=\charQ_n - \frac{1}{2}t \sum_{h=2}^{n-2}
\charQ_{h}\charQ_{n-h} + \frac{1}{2}t \charQ_{\frac{n}{2}}^{[2]},\]
which can be derived from \eqref{eq:QP}.
Extracting the coefficients of $\frac{1}{n!}m^n$ from \eqref{eq:FMchar} yields a relation between the Poincar\'e polynomials:
\begin{align*}
P_{\PP^1[n]} & =  (1+t)^2 Q_n -2t(1+t) P_n  + nt(1+t) Q_{n-1} \\
&= (1+t)^2 P_{n+1} +(n-2)t(1+t) P_n .
\end{align*}
In terms of notations in Section \ref{sec:FM}, this gives
\beq \label{eq:fmmzn}\psi_n = (1+u)^2\varphi_n + \frac{n-2}{n}u(1+u) \varphi_{n-1}. \eeq
Consequently, the asymptotic formula \eqref{eq:FMasymp} can be recovered from \eqref{eq:fmmzn} using \eqref{eq:varphiasymp}.

Furthermore, setting $m=1$ in \eqref{eq:FMchar2} yields an identity between the Poincar\'e polynomials of the quotient spaces, which can be expressed in terms of ordinary generating functions.

Let
\begin{align*} \overline{\psi}(z,u)&=1+\sum_{n\ge 1} z^n \sum_k u^k\dim H^{2k}(\PP^1[n]/\symS_n), \\
 \overline{\varphi}(z,u)&=z+\sum_{n\ge 2} z^n \sum_k u^k\dim H^{2k}(\mzno/\symS_n).\end{align*}
Note that we do not divide by $n!$. Then by \eqref{eq:FMchar2}, we obtain
\beq
\overline{\psi}= (1+\overline{\varphi})(u^2 \overline{\varphi} -u(u-1)z +1).
\eeq
Interestingly, the equation obtained here matches \eqref{3} exactly.

We expect that an argument analogous to that of Section \ref{sec:mzn} should apply to $\overline{\varphi}$. Under this unproven expectation, we expect
\[\overline{\varphi} \approx \bar\lambda(u) +\bar\mu(u) \left(1-\frac{z}{\bar\rho(u)}\right)^\frac12. \]
Since $ \overline{\psi}$ is a polynomial in $\overline{\varphi}$ and $z$, $ \overline{\psi}$ is also of the form
\[\overline{\psi} \approx A(u) +B(u) \left(1-\frac{z}{\bar\rho(u)}\right)^\frac12, \]
for some analytic functions $A$, $B$ of $u$ in a neighborhood of $u=1$. If we further assume $B(1)\ne 0$, then the Quasi-Powers Theorem should tell us that
the asymptotic distributions of $[z^n]\overline{\varphi}$ and $[z^n]\overline{\psi}$ are normal distributions whose variances are linear with the same coefficient of $n$ (determined by $\bar\rho$).

Moreover, there is an identity between the generating functions for the Poincar\'{e} polynomials of $\mzn/\symS_n$ and $\mzno/\symS_n$. See \cite[Theorem 5.11]{CKL2}, where a slightly different notation is used. In their notation, $\mathfrak{q}$ corresponds to our $\overline{\varphi}+1$. Nevertheless, the analogous arguments apply, and we therefore expect that the asymptotic distributions of their Betti numbers are normal with variances growing linearly at the same rate.

\bibliographystyle{amsplain}

\end{document}